\documentclass[12pt]{amsart}

\usepackage{amssymb,epsf}

\newtheorem{thm}{Theorem}[section] 

\newtheorem{lem}[thm]{Lemma}

\newtheorem{prop}[thm]{Proposition}
\newtheorem{rem}[thm]{Remark}

\newtheorem{observ}[thm]{Observation}

\begin{document}

\title[Orchard crossing number of prisms and ladders]{On the Orchard crossing number of prisms, ladders and other related graphs}

\author{Elie Feder and David Garber}

\address{Kingsborough Community College of CUNY, Department of Mathematics and Computer Science,
2001 Oriental Blvd., Brooklyn, NY 11235, USA}
\email{eliefeder@gmail.com, efeder@kbcc.cuny.edu}

\address{Department of Applied Mathematics, Faculty of Sciences, Holon Institute of Technology, Golomb 52,
PO Box 305, Holon 58102, Israel and (Sabbatical:) Einstein Institute of Mathematics, Hebrew University of Jerusalem, Jerusalem, Israel}
\email{garber@hit.ac.il}

\date{\today}

\begin{abstract}
This paper deals with the Orchard crossing number of some families of
graphs which are based on cycles. These include disjoint cycles, cycles which
share a vertex and cycles which share an edge.
Specifically, we focus on the prism and ladder graphs.
\end{abstract}

\maketitle

\section{Introduction}

Let $G=(V,E)$ be a graph. A {\it rectilinear drawing} $R(G)$ is an injective mapping of the vertices of $G$ to a set of points in the plane in general position (i.e. no three points are collinear), and a mapping of the edges of $G$ to straight line segments between the corresponding points.
An intersection of an edge of $R(G)$ with a straight line generated by a different pair of points of $R(G)$ is interpreted as an {\it Orchard crossing} (following the Orchard relation introduced in \cite{bacher,BaGa}). If $c(s,t)$ counts these Orchard crossings for an edge $(s,t)\in E$, then the number $c(R(G))$ of Orchard crossings of $R(G)$ is:
$$c(R(G)) = \sum _{(s,t) \in E} c(s,t);$$
note that the sum is taken only over the edges of the graph, whence $c(s,t)$
counts {\it all} the lines generated by pairs of points.

The {\it Orchard crossing number} of $G$, ${\rm OCN}(G)$, is the minimum over all rectilinear drawings $R(G)$ of $G$:
$${\rm OCN}(G) = \min _{R(G)} \{c(R(G))\}.$$

Note that the {\it maximum Orchard crossing number} of $G$ (denoted by ${\rm MOCN}(G)$ and defined as the maximal number of crossings over all drawings of $G$) is also interesting due to Proposition 2.7 in \cite{FG1} which states that ${\rm MOCN}(K_n)$ and the {\it rectilinear crossing number} $\overline{\rm cr}(K_n)$ (see \cite{aak2,F,PT}) of the complete graph $K_n$ are attained by the same realization $R(K_n)$. In general, the determination of ${\rm MOCN}(G)$ might be easier than the determination of $\overline{\rm cr}(G)$. Furthermore, the concept of the Orchard crossing number can be considered in higher dimensions too (see \cite{bacher}). In \cite{FG1,FG2}, we have computed the Orchard crossing number for some families of graphs.

\medskip

In this paper, we deal with the Orchard crossing number of
some families of graphs which are based on cycles. These include
disjoint cycles,
cycles which share a vertex and cycles which share an edge.
Note that in the context of the rectilinear crossing number,
all of these families have no crossings, since all of these graphs
are planar. Therefore,
the Orchard crossing number succeeds to differentiate between these graphs.

\medskip

Mainly, we have the following results:

\begin{thm}   \ \\
\begin{enumerate}
\item Let $G$ be the graph consisting of a disjoint union of $x$ cycles of order $n$. Then:
$${\rm OCN}(G)=n(n-2)x(x-1).$$
\item Let ${\mathcal G}$ be the family of graphs which share the property that they consist of a closed chain of $x$ cycles of order $n$, each attached to the adjacent cycle by a vertex. Then:
$$\min_{G \in \mathcal G} \{{\rm OCN}(G)\}=x(n-2)(xn-x-n).$$
\item Let ${\mathcal G}$ be the family of graphs which share the property that they consist of an open chain of $x$ cycles of order $n$, each attached to the adjacent cycle by a vertex. Then:
$$\min_{G \in \mathcal G} \{{\rm OCN}(G)\}=x(x-1)(n-1)(n-2).$$
\item Let $G$ be the graph consisting of $x$ cycles of order $3$ all attached at a common vertex. Then:
$${\rm OCN}(G)=x(x-1)^2.$$
\item Let $P_n$ be the $n$-prism graph (see its definition in Section \ref{prism}). Then for even $n>4$, we have:
$$\max\{3n(n-2),4n(n-3)\} \leq {\rm OCN}(P_n) \leq 4n(n-2).$$
For odd $n>4$, we have:
$$\max\{3n(n-2),4n(n-3)\} \leq {\rm OCN}(P_n) \leq 4n(n-2)+2.$$
\item  Let $L_n$ be the $n$-ladder graph (see its definition in Section \ref{ladder}). Then for odd $n>5$, we have:
$$\max\{3n^2-10n +7,4(n-2)(n-3)\} \leq {\rm OCN}(L_n) \leq 4(n-1)(n-2).$$
For even $n>5$, we have:
$$\max\{3n^2-10n +8,4(n-2)(n-3)\} \leq {\rm OCN}(L_n) \leq 4(n-1)(n-2).$$
\end{enumerate}
\end{thm}

\medskip

The paper is organized as follows. In section \ref{sec_obs},
we state some important observations. In Section \ref{disjoint_sec},
we deal with the graph consisting of a disjoint union of $x$ cycles of
order $n$. In Section \ref{chain_sec}, we deal with the family of graphs consisting of
a chain of $x$ cycles of order $n$ with adjacent cycles attached
to each other at one vertex.
In Section \ref{common_sec}, we deal with a graph consisting of
$x$ cycles of order $3$ all attached at a common vertex.
Section \ref{prism} deals with the $n$-prism graph and Section \ref{ladder}
deals with the $n$-ladder graph.

\section{Easy observations}\label{sec_obs}

The following observations are very easy, but will be used numerous times throughout the paper (see examples in Figure \ref{pt_and_cycle}).
\begin{observ}\label{obs}
\begin{enumerate} \ \\
\item A point inside a cycle of order $n$ contributes at least $n$ Orchard crossings, and this minimum is attained in the case that the cycle is in convex position.
\item A point outside a cycle of order $n$ contributes at least $n-2$ Orchard crossings, and this minimum is attained in the case that the cycle is in convex position.
\end{enumerate}
\end{observ}

\begin{figure}[h]
\epsfysize=4cm
\centerline{\epsfbox{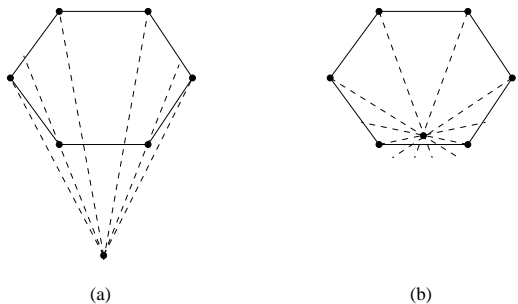}}
\caption{Part (a) illustrates $n-2$ Orchard crossings on a cycle of order $n$ determined by a point outside the cycle, while part (b) illustrates $n$ Orchard crossings on a cycle determined by a point inside an $n$-cycle.} \label{pt_and_cycle}
\end{figure}

\section{The Orchard crossing number of a disjoint union of $x$ cycles}\label{disjoint_sec}

In this section, we compute the Orchard crossing number of the graph consisting of a disjoint union of $x$ cycles of order $n$.

\begin{prop}
Let $G$ be the graph consisting of a disjoint union of $x$ cycles of order $n$. Then:
$${\rm OCN}(G)=n(n-2)x(x-1),$$
which is attained on the realization placing all the $nx$ points in convex position with the points representing each cycle adjacent to each other around the convex hull (see Figure \ref{union_cn} for an example for $n=4,x=3$).

\begin{figure}[h]
\epsfysize=4cm
\centerline{\epsfbox{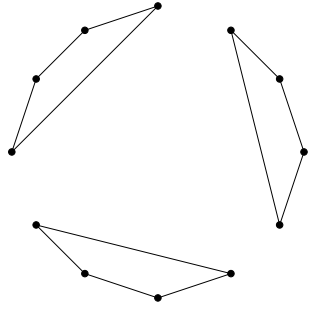}}
\caption{The best realization of a graph consisting of $3$ disjoint cycles of order $4$} \label{union_cn}
\end{figure}

\end{prop}

\begin{proof}
It is easy to see that the above-mentioned realization has indeed $n(n-2)x(x-1)$ Orchard crossings, since each of the $x$ internal chords has $n(n-2)(x-1)$ Orchard crossings (following the fact that there are $n-2$ points on one of its sides and $n(x-1)$ points on its other side). Hence: ${\rm OCN} (G) \leq n(n-2)x(x-1)$.

On the other hand, based on Observation \ref{obs}, any point separated from a cycle contributes at least $n-2$ Orchard crossings (since any point inside the cycle contributes at least $n$ crossings, and any point outside the cycle contributes at least $n-2$ crossings). For each of the $x$ cycles, we have $n(x-1)$ such points, so we have at least
$n(n-2)x(x-1)$ {\it distinct} Orchard crossings (they are distinct, because the edges of the cycles are disjoint, and the Orchard crossings are located on the edges of the cycles). Therefore, ${\rm OCN} (G) \geq n(n-2)x(x-1)$, and we are done.
\end{proof}

\begin{rem}
The same proof will work also for a disjoint union of cycles of different orders. The formula for the Orchard crossing number will be changed accordingly.
\end{rem}

\section{The Orchard crossing number of a chain of $x$ cycles}\label{chain_sec}

In this section, we compute the minimal value of the Orchard crossing number over the family of graphs consisting of closed and open chains of $x$ cycles of order $n$ with adjacent cycles attached to each other at one vertex (see Figures \ref{closed_chain_fig} and \ref{open_chain_fig} for examples for the minimal values for a closed chain and an open chain, respectively). The reason that this is indeed a {\it family} of graphs is due to the fact that the varied distance between the two points of each cycle which are shared with adjacent cycles determines different graphs.

\begin{prop}\label{closed_chain}
Let ${\mathcal G}$ be the family of graphs which share the property that they consist of a closed chain of $x$ cycles of order $n$, each attached to the adjacent cycle by a vertex. Then:
$$\min_{G \in \mathcal G} \{{\rm OCN}(G)\}=x(n-2)(xn-x-n).$$
This minimum is attained by the graph in which the two points of each cycle which are shared with another cycle are adjacent to each other. The realization of this graph which yields this minimum places all the $(n-1)x$ points in convex position with the points representing each cycle adjacent to each other around the convex hull (see Figure \ref{closed_chain_fig} for an example for $n=4,x=3$).

\begin{figure}[h]
\epsfysize=4cm
\centerline{\epsfbox{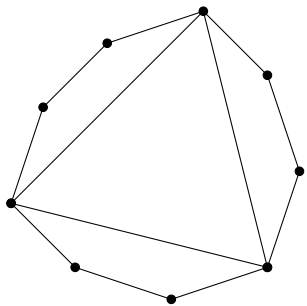}}
\caption{The best realization of a graph consisting of a closed chain of $3$ cycles of order $4$} \label{closed_chain_fig}
\end{figure}

\end{prop}

\begin{proof}
It is easy to see that the above-mentioned realization has indeed $x(n-2)(xn-x-n)$ Orchard crossings, since each of the $x$ internal chord has $(n-2)((x-1)(n-1)-1)$ Orchard crossings (following the fact that there are $n-2$ points on one of its sides and $(x-1)(n-1)-1$ points on its other side). Hence:
$$\min_{G \in \mathcal G} \{ {\rm OCN} \} (G) \leq x(n-2)(xn-x-n).$$

On the other hand, based on Observation \ref{obs}, given any graph $G \in \mathcal G$, any point separated from a cycle contributes at least $n-2$ Orchard crossings. For each of the $x$ cycles, we have $(n-1)(x-1)-1$ such points, so we have at least
$x(n-2)((n-1)(x-1)-1)$ distinct Orchard crossings (since again the {\it edges} of the cycles are disjoint).
Therefore:
$$\min_{G \in \mathcal G} \{ {\rm OCN} (G) \} \geq x(n-2)(xn-x-n),$$
and we have equality.
\end{proof}

\medskip

The situation for an open chain is almost the same, so we only formulate the result:
\begin{prop}\label{open_chain}
Let ${\mathcal G}$ be the family of graphs which share the property that they consist of an open chain of $x$ cycles of order $n$, each attached to the adjacent cycle by a vertex. Then:
$$\min_{G \in \mathcal G} \{ {\rm OCN}(G)\}=x(x-1)(n-1)(n-2).$$
This minimum is attained by the graph in which the two points of each cycle which are shared with another cycle are adjacent to each other. The realization of this graph which yields this minimum places all the $(n-1)x+1$ points in convex position with the points representing each cycle adjacent to each other around the convex hull (see Figure \ref{open_chain_fig} for an example for $n=4,x=3$).

\begin{figure}[h]
\epsfysize=4cm
\centerline{\epsfbox{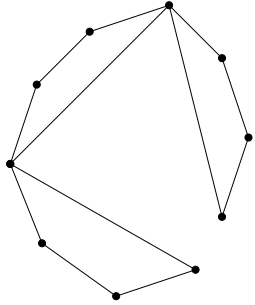}}
\caption{The best realization of a graph consisting of an open chain of $3$ cycles of order $4$} \label{open_chain_fig}
\end{figure}

\end{prop}


\section{The Orchard crossing number of $x$ cycles sharing a common vertex}\label{common_sec}

In this section, we compute the Orchard crossing number for some subcases of the family of graphs consisting of $x$ cycles sharing a common vertex.

\subsection{The case of $x$ cycles of order $3$}

Here, we treat the case of a graph consisting of $x$ cycles of order $3$, all attached at a common vertex:

\begin{prop}
Let $G$ be the graph consisting of $x$ cycles of order $3$ all attached at a common vertex. Then:
$${\rm OCN}(G)={\rm OCN}(K_{2x,1})=x(x-1)^2,$$
where $K_{2x,1}$ is the complete bipartite graph on two sets, one of $2x$ vertices and one of $1$ vertex.

This number is attained by the realization placing the common point in the center of a circle, the other $2n$ points on that circle, and the non-common points of each $3$-cycle adjacent to each other. Moreover, there is exactly one cycle of order $3$ which lies in the area bounded by the two lines generated by the central point and the two other points of each cycle  (see Figure \ref{x_c3_point} for an example for $x=4$).

\begin{figure}[h]
\epsfysize=4cm
\centerline{\epsfbox{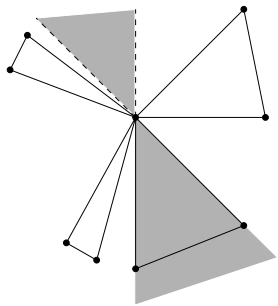}}
\caption{The best realization of a graph consisting of $x$ cycles of order $3$, all attached at a common vertex. In the gray area, there can be only one cycle (in order to have no Orchard crossings on the edges located on the convex hull).} \label{x_c3_point}
\end{figure}

\end{prop}

\begin{proof}
Note that the graph $K_{2x,1}$ is a subgraph of $G$ on the {\it same set} of vertices. Therefore:
$${\rm OCN}(K_{2x,1}) \leq {\rm OCN}(G).$$
On the other hand, the realization of $G$ presented in the formulation of the proposition has no additional Orchard crossings than those of $K_{2x,1}$ (this is due to the additional requirement on the manner of the placement of the points on the circle). Therefore, we have:
$${\rm OCN}(K_{2x,1}) \geq {\rm OCN}(G).$$
Thus, we have an equality.

Now, substituting $2x$ for $n$ in the Orchard crossing number for the graph $K_{n,1}$,
$${\rm OCN}(K_{n,1})=\frac{n(n-2)^2}{8} \qquad \hbox{(see \cite[Proposition 4.1]{FG1}),}$$
yields the claimed result.
\end{proof}

\begin{rem}
The argument of the proof {\em will not work} in the case of a graph $G$ consisting of $x$ cycles of order $n$ attached at a common vertex for $n>3$. This is because for $n>3$, the graph $K_{(n-1)x,1}$ is no longer a subgraph of $G$.
\end{rem}

\subsection{The cases of two and three cycles of order $n$}

The case of a graph $G$ consisting of two cycles of order $n$ which share a common vertex is actually a subcase of an open chain of length $2$. Thus, substituting $2$ for $x$ in Proposition \ref{open_chain} yields that:
$${\rm OCN}(G)=2(n-1)(n-2).$$

The case of a graph $G$ consisting of three cycles of order $n$ which share a common vertex is treated in the following proposition. Its proof is similar to the proof of Proposition \ref{closed_chain}.

\begin{prop}
Let $G$ be the graph consisting of three cycles of order $n$, all attached at a common vertex. Then:
$${\rm OCN}(G)=6(n-1)(n-2).$$
This number is attained by the realization placing the common point in the center of a circle, the other $3(n-1)$ points on that circle, and the non-common points of each $n$-cycle adjacent to each other. Moreover (as in the previous proposition), there is exactly one cycle of order $n$ which lies in the area bounded by two lines generated by the central point and the two points adjacent to the center in the same cycle (see Figure \ref{3_c4_point} for an example for $n=4$).

\begin{figure}[h]
\epsfysize=4cm
\centerline{\epsfbox{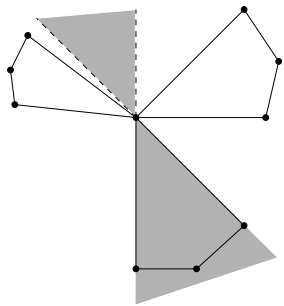}}
\caption{The best realization of a graph consisting of three cycles of order $4$, all attached in a common vertex} \label{3_c4_point}
\end{figure}

\end{prop}



\section{The Orchard crossing number of the $n$-prism}\label{prism}

The {\it $n$-prism graph}, denoted by $P_n$, is a graph consisting of $2n$ vertices and $3n$ edges organized as two cycles of order $n$, with the corresponding vertices of the cycles connected by edges (see Figure \ref{5_prism_graph} for an example for $n=5$).

\begin{figure}[h]
\epsfysize=4cm
\centerline{\epsfbox{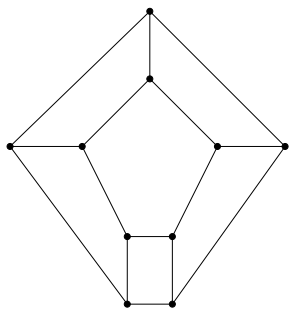}}
\caption{The 5-prism graph} \label{5_prism_graph}
\end{figure}

In this section, we deal with the Orchard crossing number of the $n$-prism.
We start by constructing drawings for the upper bound. Then, we supply
two lower bounds.

\begin{prop}
$${\rm OCN}(P_3) \leq 10\  ; \ {\rm OCN}(P_4) \leq 32. $$
For even $n>4$, we have:
$${\rm OCN}(P_n) \leq 4n(n-2).$$
For odd $n>4$, we have:
$${\rm OCN}(P_n) \leq 4n(n-2)+2.$$
\end{prop}

\begin{proof}
For the $3$-prism $P_3$ and the $4$-prism $P_4$, the upper bounds are attained by the drawings presented in Figure \ref{realize_3_4_prism}.

\begin{figure}[h]
\epsfysize=4cm
\centerline{\epsfbox{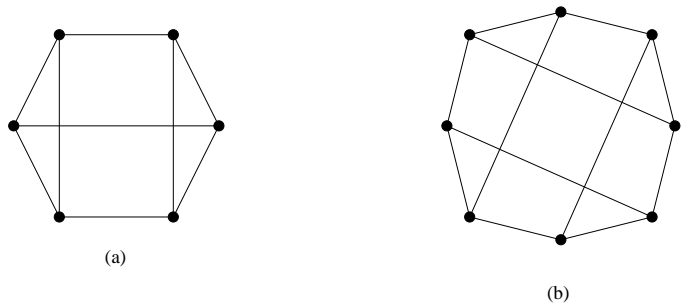}}
\caption{The realizations for the $3$-prism and the $4$-prism which yield the upper bounds} \label{realize_3_4_prism}
\end{figure}

For $n>4$, the upper bound for the Orchard crossing number of the
$n$-prism graph is attained by the following construction.
Draw a regular $2n$-gon. Starting from any point, color the
vertices white and black (corresponding to the two cycles
of order $n$) as follows: two whites, two blacks, two whites,
two blacks, etc. For the case of odd $n$,
the last two vertices will be one white and one black.

Now connect the white (resp. black) vertices around the $2n$-gon so
that each white (resp. black) vertex is connected to its closest clockwise white (resp. black) neighbor. For even $n$, we complete the graph by filling in the remaining edges around the $2n$-gon. For odd $n$, in completing the graph we do not draw the edge between the single white vertex and the single black vertex (see Figure \ref{realize_prism} for both cases).

\begin{figure}[h]
\epsfysize=4cm
\centerline{\epsfbox{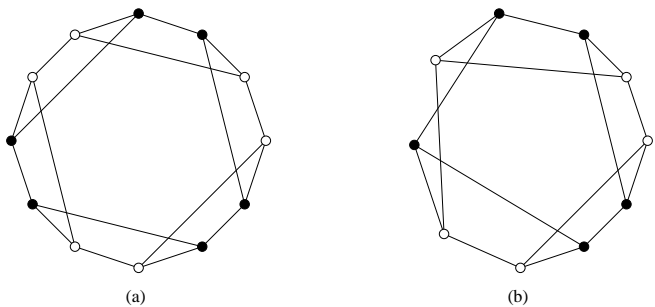}}
\caption{The realizations which yield the upper bounds for the $n$-prism graph; part (a) for even $n$ and part (b) for odd $n$} \label{realize_prism}
\end{figure}

In these drawings, all edges on the $2n$-gon have no Orchard crossings at all
(since there are no points inside the convex hull).  In the case of even $n$,
we have $n$ additional edges, each contributing $2(2n-4)$ Orchard crossings,
yielding the claimed upper bound of $4n(n-2)$. For the case of odd $n$, we have
$n-1$ additional edges, each contributing $2(2n-4)$ Orchard crossings and two edges, each contributing $2n-3$ Orchard crossings. This yields $2(n-1)(2n-4) + 2(2n-3)$ Orchard crossings, which is equal to the claimed upper bound.
\end{proof}

Now, we move to the lower bounds:

\begin{prop}
$$\max\{3n(n-2),4n(n-3)\} \leq {\rm OCN}(P_n).$$
\end{prop}

\begin{proof}
We show that the Orchard crossing number is larger than (or equal to)
both numbers appearing in the formulation of the result.

We start with the bound $3n(n-2)$. For each of the $n$ cycles of order $4$,
each of the $2n-4$ points separate from it contributes two Orchard crossings
(by Observation \ref{obs}) on the edges of these cycles.
For each of the two cycles of order $n$, each of the $n$ points separate from it
contributes $n-2$ Orchard crossings on the edges of these cycles (again, by Observation \ref{obs}). This gives a total of
$$2n(2n-4) + 2n(n-2) = 6n(n-2)$$
Orchard crossings.

However, in these cycles, each edge of the graph is counted exactly twice. The reason for this is
as follows: the edges shared by two cycles  of order $4$ are counted twice, once for each of these cycles.
These edges are not in the cycles of order $n$. On the other hand, the edges
in the cycles of order $n$ are the edges not shared by two cycles of order $4$,
but are only on one cycle of order $4$. Thus, we have to divide $6n(n-2)$ by $2$ and we get the required lower bound.

\medskip

Now, we pass to the second bound $4n(n-3)$. The $n$-prism $P_n$ contains $n$ cycles of order $4$. As each cycle has $2n-4$ vertices separated from the cycle, we know by Observation \ref{obs} that each of these cycles determines $2(2n-4)$ Orchard crossings on the edges of that cycle. However, there are $n$ edges which are counted in two cycles. Therefore, we must subtract the number of Orchard crossings on these edges which will be overcounted.

To do this, we consider two cycles of order $4$ which share a common edge. For each of these cycles, we count crossings determined by pairs of points, one of which is in the cycle and the other of which is not in the cycle. The only pairs which will be counted in both cycles and will therefore be overcounted in our count of  Orchard crossings are those with the following property: both points are not on the shared edge, and one point is in each of the two cycles of order $4$ (see Figure \ref{2_attached_c4}). Since each cycle has two points not on the shared edge, we have $4$ such pairs and therefore $4$ possible overcounted Orchard crossings for each of the $n$ shared edges in the graph. Thus, we have a lower bound of $2n(2n-4)-4n= 4n(n-3)$ as claimed.

\begin{figure}[h]
\epsfysize=4cm
\centerline{\epsfbox{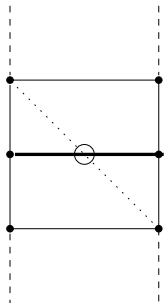}}
\caption{In the presented segment of $P_n$, the Orchard crossing (which is encircled) on the bold edge generated by the dotted line is counted twice, once for each of the upper and lower cycles of order $4$.} \label{2_attached_c4}
\end{figure}

\end{proof}

\begin{rem}\ \\
\begin{enumerate}
\item Note that for $n \leq 6$, $3n(n-2) \geq 4n(n-3)$ and for $n>6$, $3n(n-2) < 4n(n-3)$.
\item Moreover, observe that in the realizations which yield the upper bound (see Figure \ref{realize_prism}), there are no overcounts at all (since the shared edges between two adjacent cycles of order $4$ are placed on the convex hull, and hence have no crossings on them).
\item It is worth mentioning that {\em asymptotically}, the lower bound and the upper bound are the same.
\end{enumerate}
\end{rem}

\section{The Orchard crossing number of the $n$-ladder}\label{ladder}

The {\it $n$-ladder graph}, denoted by $L_n$, is a graph consisting of $2n$ vertices and $3n-2$ edges organized as a ladder, i.e. we have a sequence of $n-1$ cycles of order $4$ where any two consecutive cycles share an edge (see Figure \ref{5_ladder_graph} for an example for $n=5$). Note that the $n$-ladder graph has two edges less than the $n$-prism graph.

\begin{figure}[h]
\epsfysize=4cm
\centerline{\epsfbox{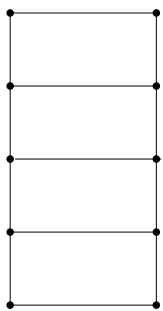}}
\caption{The $5$-ladder graph} \label{5_ladder_graph}
\end{figure}

In this section, we deal with the Orchard crossing number of the $n$-ladder.

\subsection{The $3$-ladder and $4$-ladder}

We start with the exact values of the Orchard crossing numbers of the $3$-ladder and the $4$-ladder.

\begin{lem}
${\rm OCN}(L_3)=4$.
\end{lem}

\begin{proof}
Consider one cycle of order $4$. It has two points outside of it. By Observation \ref{obs}, each of these two points generates at least two Orchard crossings on the edges of the cycle. Therefore, we have at least $4$ Orchard crossings. This bound is realized by the relevant drawing in Figure \ref{realize_L3_L4_L5}.

\begin{figure}[h]
\epsfysize=4cm
\centerline{\epsfbox{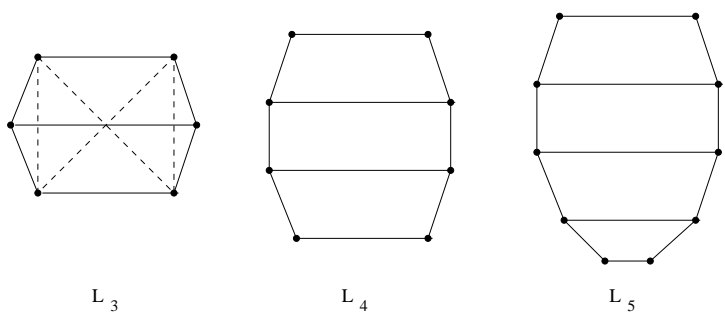}}
\caption{The best realizations of $L_3$, $L_4$ and $L_5$} \label{realize_L3_L4_L5}
\end{figure}

\end{proof}

\begin{lem}
${\rm OCN}(L_4)=16$.
\end{lem}

\begin{proof}
Consider two disjoint cycles of order $4$. Each has $4$ points outside of it. By Observation \ref{obs}, each of these $4$ points generates at least two Orchard crossings on the edges of the cycle. Therefore, we have at least $16$ Orchard crossings. This bound is realized by the relevant drawing in Figure \ref{realize_L3_L4_L5}.
\end{proof}

\subsection{Upper and lower bounds for the $n$-ladder, $n\geq 5$}

We start with the upper bound.

\begin{prop}
$${\rm OCN}(L_5) \leq 40.$$
For $n>5$, we have:
$${\rm OCN}(L_n) \leq 4(n-1)(n-2).$$
\end{prop}

\begin{proof}
The drawing for the upper bound for the Orchard crossing number of the $5$-ladder $L_5$ is attained by the relevant drawing in Figure \ref{realize_L3_L4_L5}.

The drawing for the upper bound for the Orchard crossing number of the $n$-ladder for $n>5$ is very similar to that of the $n$-prism (see Section \ref{prism}).
The difference between these graphs is that the $n$-ladder does not contain any cycles of order $n$. Thus, for the drawing of the $n$-ladder, we follow the drawing of the $n$-prism, but we do not connect the final edge of the white (resp. black) cycle. For the case of even $n$, the edges which are omitted must be in the same cycle of order $4$ in $P_n$. For the case of odd $n$, the edges which are omitted are the edges connecting two white (resp. black) points which have a single black (resp. white) point in between them (see Figure \ref{realize_ladder} for both cases).

\begin{figure}[h]
\epsfysize=4cm
\centerline{\epsfbox{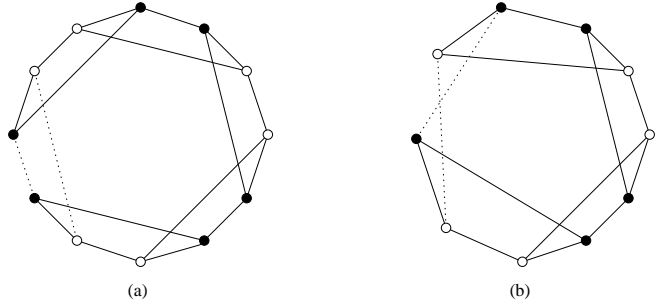}}
\caption{The realizations which yield the upper bound for the $n$-ladder; part (a) for even $n$ and part (b) for odd $n$. The dotted edges are the two omitted edges from the realizations of the upper bound for the $n$-prism (see Figure \ref{realize_prism}).} \label{realize_ladder}
\end{figure}

To count the Orchard crossings in this drawing, we again notice that all edges around the $2n$-gon have no Orchard crossings. For both even and odd $n$, we have $n-1$ edges which have $2(2n-4)$ Orchard crossings each. This yields an upper bound of $4(n-1)(n-2)$ Orchard crossings as claimed.
\end{proof}

\begin{rem}
The drawing of $L_6$ can also be drawn similar to the way $L_3,L_4$ and $L_5$ were drawn (both drawings of $L_6$ have the same number of Orchard crossings).
\end{rem}

Now, we move to the lower bounds:
\begin{prop}
For odd $n>5$, we have:
$$\max\{3n^2-10n +7,4(n-2)(n-3)\} \leq {\rm OCN}(L_n).$$
For even $n>5$, we have:
$$\max\{3n^2-10n +8,4(n-2)(n-3)\} \leq {\rm OCN}(L_n).$$
\end{prop}

\begin{proof}
The method of attaining lower bounds is similar to the case of the $n$-prism.
We start with the first bound. We find a set of cycles which counts each edge of the graph twice. We determine the number of Orchard crossings which must exist on the edges of these cycles, and we divide by $2$ to get a lower bound for ${\rm OCN}(L_n)$.

We first consider the case of even $n$. We start with two cycles of order $n+2$ (see Figure \ref{overlap_cn}): one cycle contains the top rung of the ladder (the {\it upper} cycle) and the other cycle contains the bottom rung of the ladder (the {\it lower} cycle). We also consider $n-2$ of the $n-1$ cycles of order $4$ contained in the graph. We choose these cycles so that we exclude the only $4$-cycle which contains {\em both} an edge of the upper cycle and an edge of the lower cycle. This set of cycles counts each edge of the graph twice.

\begin{figure}[h]
\epsfysize=4cm
\centerline{\epsfbox{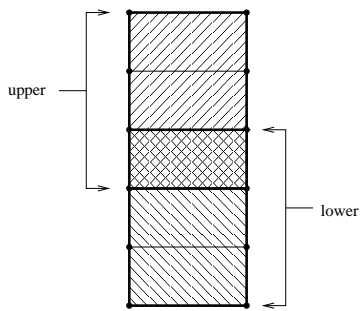}}
\caption{The upper and lower cycles of order $n+1$ (the upper one has diagonal lines with a positive slope)} \label{overlap_cn}
\end{figure}

We now count the necessary Orchard crossings on these edges. By Observation \ref{obs}, for each cycle of order $n+2$, we have at least\break $n(n-2)$ Orchard crossings and for each cycle of order $4$, we have at least $2(2n-4)$ Orchard crossings. This gives us a total of $$\frac{1}{2}[2(n-2)(2n-4)+2n(n-2)] = 3n^2-10n+8$$
Orchard crossings as claimed.

We consider the case of odd $n$: We start with one cycle of order $n+3$ and one cycle of order $n+1$. The cycle of order $n+3$ contains the top rung of the ladder and the cycle of order $n+1$ contains the bottom rung of the ladder. We also consider $n-2$ of the $n-1$ cycles of order $4$ contained in the graph. We choose these cycles so that we exclude the only $4$-cycle which contains {\em both} an edge of the upper cycle and an edge of the lower cycle. This set of cycles counts each edge of the graph twice.

We now count the necessary Orchard crossings on these edges. By Observation \ref{obs}, for the cycle of order $n+3$ we have at least\break $(n+1)(n-3)$ Orchard crossings,  for the cycle of order $n+1$ we have at least $(n-1)^2$ Orchard crossings, and for each cycle of order $4$ we have at least $2(2n-4)$ Orchard crossings. This gives us a total of
$$\frac{1}{2}[2(n-2)(2n-4)+(n-1)^2+(n+1)(n-3)] = 3n^2-10n+7$$
Orchard crossings as claimed.

\medskip

Now we move to the second bound $4(n-2)(n-3)$. The $n$-ladder $L_n$ contains $n-1$ cycles of order $4$. As each cycle has $2n-4$ points separated from the cycle, we know by Observation \ref{obs} that each of these cycles determines $2(2n-4)$ Orchard crossings on the edges of the cycle. However, there are $n-2$ edges which are counted in two adjacent cycles. We therefore must subtract the number of Orchard crossings which may be overcounted. As in the case of the $n$-prism $P_n$, we have $4$ possible overcounted Orchard crossings for each of the $n-2$ shared edges in the graph. Thus, we have a lower bound of $2(n-2)(2n-4)-4(n-2)= 4(n-2)(n-3)$ as claimed.
\end{proof}

\begin{rem} \ \\
\begin{enumerate}
\item Note that for $n \leq 8$, $3n^2-10n +8 \geq 4(n-2)(n-3)$
and for $n>8$, $3n^2-10n +8 < 4(n-2)(n-3)$.
\item As before, observe that in the realizations which yield the upper bound (see Figure \ref{realize_ladder}), there are no overcounts at all (since the shared edges between two adjacent cycles of order $4$ are placed on the convex hull, and hence have no crossings on them).
\item It is worth mentioning that as in the case of the prism graph, the lower bound and the upper bound are {\em asymptotically} the same.
\end{enumerate}

\end{rem}



\begin{thebibliography}{99}
\bibitem{aak2} O. Aichholzer, F. Aurenhammer and H. Krasser,
    {\it On the crossing number of complete graphs}, in: Proc. 18th
    Ann. ACM Symp. Computational Geometry, 19--24, Barcelona, Spain, 2002.
\bibitem{AFH} M. Alpert, E. Feder and H. Harborth, {\it The maximum of the maximum rectilinear crossing numbers of $d$-regular graphs of order $n$},  Electron. J. Combin.  {\bf 16}(1)  (2009), Research Paper 54, 16 pp.
\bibitem{bacher} R. Bacher, {\it Le cocycle du verger}, C. R. Acad. Sci. Paris,
    Ser. I, {\bf 338}(3) (2004), 187--190.
\bibitem{BaGa} R. Bacher and D. Garber,  {\it The Orchard relation
    of planar configurations of points I}, Geombinatorics {\bf 15}(2) (2005), 54--68.
\bibitem{FG1} E. Feder and D. Garber,  {\it The Orchard crossing number of an abstract graph}, Proceedings of the 40th Southeastern International Conference on Combinatorics, Graph Theory and Computing, Cong. Numer. {\bf 197} (2009), 3--19.
\bibitem{FG2} E. Feder and D. Garber,  {\it On the Orchard crossing number of the complete bipartite graphs $K_{n,n}$}, submitted.
\bibitem{F} S. Felsner, {\it Geometric graphs and arrangements},
    Advanced Lectures in Mathematics. Friedr. Vieweg \& Sohn, Wiesbaden, 2004.
\bibitem{PT} J. Pach and G. T\'oth, {\it Thirteen problems on
    crossing numbers}, Geombinatorics {\bf 9} (2000), 194--207.
\end{thebibliography}
\end{document}